\documentclass[12pt]{article}
\usepackage[backref]{hyperref}
\usepackage{epsfig}
\usepackage{float}
\usepackage{graphicx}
\usepackage{subfigure}
\usepackage{tikz}
\usetikzlibrary{backgrounds,patterns}
\usepackage[normalem]{ulem}
\usepackage[margin=3cm]{geometry}
\usepackage{color}
\usepackage{enumitem}
\usepackage{amsfonts}
\newlist{enum*}{enumerate*}{2}%
\setlist[enum*]{label={(\roman*)}}
\newlist{enum}{enumerate}{2}%
\setlist[enum]{label=(\alph*), nosep}

\usepackage{amsmath, amssymb, latexsym, amsthm, color}

\newtheorem{question}{Question}
\newtheorem{theorem}{Theorem}

\newtheorem{claim}{Claim}

\newtheorem{lemma}[theorem]{Lemma}

\newtheorem{definition}{Definition}

\begin{document}
	
	\title{\bf Girth and $\lambda$-choosability of graphs}
	\author{Yangyan Gu \thanks{Department of Mathematics, Zhejiang Normal University. Email: yangyan@zjnu.edu.cn} \and Xuding Zhu\thanks{Department of Mathematics, Zhejiang Normal University,  China.  E-mail: xdzhu@zjnu.edu.cn. Grant Number: NSFC 11971438,12026248, U20A2068.}}
	\date{\today}
	\maketitle
	
\begin{abstract}
	Assume $ k $ is a positive integer, $ \lambda=\{k_1,k_2,...,k_q\} $ is a partition of $ k $ and $ G $ is a graph. A $\lambda$-assignment of $ G $ is a $ k $-assignment $ L $ of $ G $ such that the colour set $ \bigcup_{v\in V(G)} L(v) $ can be partitioned into $ q $ subsets $ C_1\cup C_2\cup\cdots\cup C_q $ and for each vertex $ v $ of $ G $, $ |L(v)\cap C_i|=k_i $. We say $ G $ is $\lambda$-choosable if for each $\lambda$-assignment $ L $ of $ G $, $ G $ is $ L $-colourable. In particular,  if $ \lambda=\{k\} $, then $\lambda$-choosable is the same as $ k $-choosable, if $ \lambda=\{1, 1,...,1\} $, then $\lambda$-choosable is equivalent to $ k $-colourable. For the other partitions of $ k $ sandwiched between $ \{k\} $ and $ \{1, 1,...,1\} $ in terms of refinements, $\lambda$-choosability reveals a complex hierarchy of colourability of graphs. 
	Assume $\lambda=\{k_1, \ldots, k_q\} $ is a partition of $ k $ and $\lambda' $ is a partition of $ k'\ge k $. We write $ \lambda\le \lambda' $ if there is a partition $\lambda''=\{k''_1, \ldots, k''_q\}$ of $k'$ with $k''_i \ge k_i$ for $i=1,2,\ldots, q$ and 
	 $\lambda'$ is a refinement of   $\lambda''$. It follows from the definition that if $ \lambda\le \lambda' $, then every $\lambda$-choosable graph is $\lambda'$-choosable. It was proved in [X. Zhu, A refinement of choosability of graphs, J. Combin. Theory, Ser. B 141 (2020) 143 - 164] that the converse is also true. This paper strengthens this result and proves that for any $ \lambda\not\le \lambda' $, for any integer $g$, there exists a graph of girth at least $g$ which is $\lambda$-choosable but not $\lambda'$-choosable.
		
\end{abstract}

\section{Introduction}

A proper $ k $-colouring of a graph $ G $ is a colouring $ c : V(G)\rightarrow \{1,2,\ldots, k\} $ such that adjacent vertices receive different colours. The \emph{chromatic number} of $ G $ is the minimum integer $ k $ such that $ G $ has a proper $ k $-colouring. The {\em girth} of $G$ is the smallest length of cycles in $G$. If $G$ has girth $g$, then for any vertex $v$, the subgraph of $G$ induced by vertices at distance at most $g/2-1$ from $v$ is a tree. Hence large girth graphs  are ``locally'' 2-colourable. 
A natural question is whether locally $2$-colourable graphs can have large chromatic number. This question was answered in affirmative by Erd\H{o}s \cite{Erdos}:  For any positive integers $g,k$, there exists a  graph $G$ of   girth at least $g$ and   chromatic number at least $k$. This classical result is one of the most influential results in graph theory and has been generalized in many different ways. We may view the chromatic number as a scale that measures certain complexity of graphs. Erd\H{o}s'  result assures the existence of large girth graphs with given complexity with respect to this scale. By considering different measurements for graphs, one obtains various generalizations of this result. For example,  by consider the partial order of graph homomorphisms, it was proved in  \cite{JZ} that for any core graph $H$ and integers $g,t$, there exists a graph $G$ of girth at least $g$ such that homomorphism from $G$ to any   graph $H'$ of order at most $t$ are   composition of a homomorephism from $G$ to   $H$ and a homomorphism from $H$ to $H'$. 

This paper generalizes Erd\H{o}s' result with respect to a new measurement of colourability of graphs, which is a generalization of list colouring of graphs.  
An assignment of a graph $ G $ is a mapping $ L $ which assigns to each vertex $ v $ of $ G $ a set $ L(v) $ of permissible colours. A proper $ L $-colouring of $ G $ is a proper colouring $ f $ of $ G $ such that for each vertex $ v $ of $ G $, $ f(v)\in L(v) $. We say $ G $ is $ L $-colourable if $ G $ has a proper $ L $-colouring. A $ k $-assignment of $ G $ is a assignment $ L $ with $ |L(v)|=k $ for each vertex $ v $. We say $ G $ is $ k $-choosable if $ G $ is $ L $-colourable for any $ k $-assignment $ L $ of $ G $. The \emph{choice number} of $ G $ is the minimum integer $ k $ such that $ G $ is $ k $-choosable.

The concept of list colouring was introduced by Erd\H{o}s, Rubin and Taylor \cite{Erdos}, and independently by Vizing \cite{Viz} in the 1970’s, and provides a useful tool in many inductive proofs for upper bounds for the chromatic number of graphs, and motivates many challenging problems. There is a big gap between $ k $-colourability and $ k $-choosability. In particular, bipartite graphs can have arbitrary large choice number. A refinement of   the concept of choosability was introduced in \cite{Zhu}, which puts $ k $-choosability and $ k $-colourability in a same framework and considers a much more complex hierrachy of colourability of graphs. 

\begin{definition}
	A partition of a positive integer $ k $ is a finite multiset $ \lambda = \{k_1, k_2,..., k_q\} $ of positive integers with $  k_1+k_2+...+ k_q=k $. Each integer $ k_i \in \lambda $ is called a part of $\lambda$.
\end{definition}

\begin{definition}
	Assume $ \lambda = \{k_1, k_2,..., k_q\} $ is a partition of $ k $ and $ G $ is a graph. A $\lambda$-assignment of $ G $ is a $ k $-assignment $ L $ of $ G $ in which the colours in $ \cup_{x\in V(G)}L(x) $ can be partitioned into sets $ C_1, C_2,...,C_q $ so that for each vertex $ x $ and for each $ 1 \le  i \le  q $, $ |L(x)\cap C_i|=k_i $. Each $ C_i $ is called a colour group of $ L $. We say $ G $ is $\lambda$-choosable if $ G $ is $ L $-colourable for any $\lambda$-assignment $ L $ of $ G $.
\end{definition}

  Assume $\lambda$ and $\lambda'$ are two partitions of $ k $. We say $ \lambda' $ is a \emph{refinement} of $\lambda$ if $\lambda'$ is obtained from $\lambda$ by replacing some parts
  of $\lambda$ by partitions of 
  these parts. For example, $ \lambda'= \{2, 3, 4\} $ is a refinement of $ \lambda = \{4, 5\} $. It follows from the definition that if $\lambda'$ is a refinement of $\lambda$, then every $\lambda'$-assignment of a graph $ G $ is also a $\lambda$-assignment of $ G $. Hence every $\lambda$-choosable graph is $\lambda'$-choosable.

It is easy to see that if $ \lambda =\{1,1,..., 1\} $   consists of $ k $ copies of $ 1 $, then $\lambda$-choosable is the same as $ k $-colourable. On the other hand,   $ \{k\} $-choosable is the same as $ k $-choosable. So $\lambda$-choosability puts $ k $-colourability and $ k $-choosability of graphs under a same framework, and $\lambda$-choosability for those partitions $ \lambda $ of $ k $  sandwiched between $ \{k\} $ and $ \{1,1,...,1\} $ (in terms of refinements) reveal a complicated hierarchy of colourability of graphs.

\begin{definition}
 Assume $\lambda=\{k_1, \ldots, k_q\} $ is a partition of $ k $ and $\lambda' $ is a partition of $ k'\ge k $. We write $ \lambda\le \lambda' $ if there is a partition $\lambda''=\{k''_1, \ldots, k''_q\}$ of $k'$ with $k''_i \ge k_i$ for $i=1,2,\ldots, q$ and 
 $\lambda'$ is a refinement of   $\lambda''$.
\end{definition}

For example, $ \lambda= \{2, 2\} $ is a partition of $ 4 $, and $ \lambda'= \{1, 1, 1, 3\} $ is a partition of $ 6 $. Let $ \lambda''= \{2, 4\} $. Then $\lambda''$ is obtained from $\lambda$ by increasing one part of $\lambda$ by $ 2 $, and $\lambda'$ is a refinement of $\lambda''$. Hence $ \lambda\le \lambda' $

If $\lambda''$ is obtained from $\lambda$ by increasing some of parts of $\lambda$, then certainly every $\lambda$-choosable graph is $\lambda''$-choosable. If $\lambda'$ is a refinement of $\lambda''$, then every $\lambda''$-choosable graph is $\lambda'$-choosable. Therefore if $ \lambda\le \lambda' $, then every $\lambda$-choosable graph is $\lambda'$-choosable. It was proved in \cite{Zhu} that if 
$ \lambda\not\le \lambda' $, then there exists 
 a graph which is $\lambda$-choosable but not $\lambda'$-choosable.

\begin{theorem} \label{thm-Zhu} \cite{Zhu}
If $ \lambda\le\lambda' $, then every $ \lambda $-choosable graph is $ \lambda' $-choosable, and conversely, if every $ \lambda $-choosable graph is $ \lambda' $-choosable, then $ \lambda\le\lambda' $.
\end{theorem}

In this paper, we prove the following result, which strengthens Theorem \ref{thm-Zhu}, and generalizes Erd\H{o}s' result to the setting of $\lambda$-choosability of graphs.

\begin{theorem}\label{key}
	For any positive integer $ g $ and  $ \lambda\not \le \lambda' $, there exists a graph $ G $ of girth at least $ g $ which is $ \lambda $-choosable but not $ \lambda' $-choosable.
\end{theorem}

\section{Proof of Theorem \ref{key}}

The proof of Theorem \ref{key} uses basic  probabilitic method.   One new ingredient in the proof is to split  vertices of a large girth graph appropriately and then add copies of some other graphs and ensure that the resulting graph still has large girth and some other required properties of a random graph.  

In the calculations in our proof, the following three inequalities involving binomial coefficients will be used:
\begin{itemize}
	\item[(1)] $\binom{a}{b}\le \left(\frac{ea}{b}\right)^b$;
\end{itemize}
If $ 0\le x< b $ and $ b+x< a  $, then
\begin{itemize}
	
	\item[(2)] $ \binom{a-x}{b}\binom{a}{b}^{-1}\le \left(\frac{a-b}{a}\right)^x<e^{-bx/a} ;$
	
	\item[(3)] $ \binom{a-x}{b-x}\binom{a}{b}^{-1}\le \left(\frac{b}{a}\right)^x .$
	
\end{itemize}

\begin{lemma}\label{G}
For any positive integers $ k,g,t $ and $0 < \epsilon < 1/4g$, there exists a $ k $-partite graph  $ G_0 $ with partite sets $ V_1,V_2,\ldots,V_k $, which has the following properties:

\begin{itemize}
\item[1.] All the parts have the same size, say $|V_i|=n$. 
\item[2.] The girth of $ G_0 $ is at least $ g $, 
\item[3.] For any $ 1\le i,j \le k $ with $ i\neq j $ and any subsets $ A \subseteq V_i , B \subseteq V_j $ with $ |A|,|B|\ge  \lfloor n/t \rfloor $, there are at least $ \frac{1}{2}n^{1+\epsilon} $ edges between $ A $ and $ B $.
\end{itemize} 

\end{lemma}

\begin{proof}
Let $ F $ be a complete $ k $-partite graph with partite set $ V_1,V_2,...,V_k $ and every part has size $ n $. Let $ q=\frac{k(k-1)}{2} $, then $ F $ has $ qn^2 $ edges. Let $ \mathcal{G} $ be the set of all subgraphs $ G $ of $ F $ with $ m=\lfloor qn^{1+2\epsilon} \rfloor $ edges. Then $ |\mathcal{G}|=\binom {qn^2}{m} $. In the following, $ n $ is assumed to be sufficiently large. We consider $ \mathcal{G} $ as a probability space with each member occurring with the same probability $ 1/|\mathcal{G}| $.

\begin{claim}\label{1}
The expected number of cycles of length less than $ g $ in a graph $ G\in \mathcal{G} $ is bounded by $ n^{-\epsilon}n^{2g\epsilon} $. Thus asymptotically almost all graphs from $ \mathcal{G} $ have at most $ n^{2g\epsilon} $ cycles of length $ \le g-1 $.
\end{claim}

\begin{proof}
	The expected number of cycles $ C_l $ of length $ l $ in a graph $ G\in \mathcal{G} $ is at most
	$$
	N_l=\binom{kn}{l}\frac{l!}{2l}\binom{qn^2-l}{m-l}\binom{qn^2}{m}^{-1}.$$
	By inequality (3), 
	$$\binom{qn^2-l}{m-l}\binom{qn^2}{m}^{-1}\le \left(\frac{m}{qn^2}\right)^l.
	$$
	Since $ m\le qn^{1+2\epsilon} $,
	$$ N_l\le \binom{kn}{l}\frac{l!}{2l}\left(\frac{m}{qn^2}\right)^l<\left(\frac{km}{qn}\right)^l\le k^ln^{2\epsilon l} .$$
	Therefore
	$$ \sum_{l=3}^{g-1}N_l<(g-3)k^{g-1}n^{2(g-1)\epsilon}<n^{\epsilon} n^{2(g-1)\epsilon}=n^{-\epsilon}n^{2g\epsilon} .$$
	Here we assume that $ n $ is large enough so that $ n^\epsilon>(g-3)k^{g-1} $.
	
	This implies that if $ \mathcal{G}_1 $ is the set of all graphs $ G\in \mathcal{G} $ with at most $ n^{2g\epsilon} $ cycles of length less than $ g $, then $ |\mathcal{G}_1|\ge (1-n^{-\epsilon})|\mathcal{G|} $.
	
\end{proof}

For a graph $G \in \mathcal{G}_1$, by deleting one edge from each cycle of length at most $g-1$ (and deleting at most $n^{2g\epsilon}$ edges in total), we obtain a graph of girth at least $g$. 

\begin{claim}\label{2}
Asymptotically almost all graphs from $ \mathcal{G} $ has the property that for any $ 1\le i,j \le k $ with $ i\neq j $ and any subsets $ A \subseteq V_i , B \subseteq V_j $ with $ |A|,|B|=\lfloor n/t\rfloor $, there is at least $ n^{1+\epsilon} $ edges between $ A $ and $ B $.
\end{claim}

\begin{proof}
	For an integer $ s\le n^{1+\epsilon} $, denote by $ M(s) $ the expected number (in a graph $ G\in \mathcal{G} $) of pairs $ A\subseteq V_i, B\subseteq V_j $ with $ i\neq j $ such that $ |A|=|B|=\lfloor n/t\rfloor  $, and there are exactly $ s $ edges connecting $ A $ and $ B $. Then 
	$$ M(s)=q\binom{n}{\lfloor \frac{n}{t} \rfloor}^2 \binom{\lfloor \frac{n}{t}\rfloor^2} {s} \binom {qn^2-\lfloor \frac{n}{t}\rfloor^2}{m-s}\binom {qn^2}{m}^{-1}. $$
	Replacing $ \binom {qn^2-\lfloor \frac{n}{t}\rfloor^2}{m-s} $ by $ \binom {qn^2-\lfloor \frac{n}{t}\rfloor^2}{m} $, applying inequalities (1),(2) we have
	
	$$ M(s)<q(et)^{2n/t}\left(\frac{n}{t}\right)^{2s}e^{-(1/t^2)n^{1+2\epsilon}}. $$
	Assume  $ n $ is large enough so that
	$$ e^{-(1/2t^2)n^{1+2\epsilon}}q(et)^{2n/t}<1 .$$
	Then $$ M(s)<n^{2s} e^{-(1/2t^2)n^{1+2\epsilon}}. $$
	Hence,
	$$\sum_{s<n^{1+\epsilon}} M(s)<\exp(-(1/2t^2)n^{1+2\epsilon}+3n^{1+\epsilon}\log n)<\exp(-(1/4t^2)n^{1+2\epsilon})<e^{-n}.$$

\end{proof}

 Combining Claim \ref{1} and Claim \ref{2}, we have a $ k $-partite graph $ G_0 $ of girth at least $ g $, each part $V_i$ has $n$ vertices, and for any $ 1\le i,j \le k $ with $ i\neq j $ and any subsets $ A \subseteq V_i , B \subseteq V_j $ with $ |A|,|B|\ge \lfloor n/t\rfloor $, there is at least $ n^{1+\epsilon}-n^{2g\epsilon}>\frac{1}{2}n^{1+\epsilon} $ edges between $ A $ and $ B $. This completes the proof of Lemma \ref{G}.
\end{proof}

We shall construct the graph $G$ in Theorem \ref{key} by using $G_0$ as a base. Another gadget needed for the construction of $G$ is the following   result, which was proved by   Kostochka and Ne\v {s}et\v{r}il \cite{KN}.

\begin{theorem}\cite{KN}
	For any positive integers $r, k, g$, there is an $r$-uniform  $k$-degenerate hypergraph with girth at least $ g $ that is not $ k $-colourable. In particular, there is a   $k$-degenerate  graph with girth at least $ g $ that is not $ k $-colourable.
\end{theorem}

Assume $ \lambda = (k_1,k_2,..., k_q)$. For $i=1,2,\ldots, q$, let $J_i$ be a $(k_i-1)$-degenerate graph of girth $g$ which is not $(k_i-1)$-colourable. 
  By adding isolated vertices, we may assume that all $J_i$ have the same number of vertices, say $|V(J_i)|=r$. Let $[r]=\{1,2,\cdots,r\} $. We shall split each vertex $v$ of $G_0$ into a set $S_v = \{v\} \times [r]$ of $r$ vertices  and distribute the edges incident with $v$ to these $r$ vertices uniformly randomly. For each vertex $v \in V_i$,  we add a copy of $J_i$ with vertex set $S_v$.  We shall show that with positive probability, the resulting random graph has some nice property (stated in Lemma \ref{lem-split} below).
  
  Let $G$ be such a resulting graph. 
Let $G_i$ be the subgraph of $G$ induced by $V_i \times [r]$. So $G_i$ consists of $n$ vertex disjoint copies of $J_i$. Hence $G_i$ is $(k_i-1)$-degenerate, and is $k_i$-choosable. As a consequence the graph $G$ is $\lambda$-choosable. 

To see that $G$ has girth at least $g$, let $C$ be a cycle in $G$. If $C$ is contained in one copy of $J_i$ for some $i$, then $C$ has length at least $g$, as $J_i$ has girth at least $g$.  For any other cycle $ C $ in $ G $, contracting each copy of $ J_i $ to a single vertex yields a closed walk $ C' $ in $G_0$. Since there is at most one edge between a copy of $J_i$ and a copy of $J_{i'}$  
in $ G $, each edge is used only once in $ C' $. Hence $ C' $ contains a cycle in $G_0$, which has length at least $ g $, as $G_0$ has girth at least $g$. So 
  $ C $ has length at least $ g $ and $G$ has girth at least $g$.

Assume $ \lambda'=(k_1',k_2',...,k_p') $ and $\lambda \not\le \lambda'$. 
We shall prove that when constants $n,t$ are chosen approriately, then $G$ is not $\lambda'$-choosable.

For this purpose, we need to show that we can split each vertex $v$ of $G_0$ into a set $S_v = \{v\} \times [r]$ of $r$ vertices, so that the resulting graph $G'$ has some nice properties.

\begin{lemma}
	\label{lem-split}
	Let $G_0$ be the graph as in Lemma \ref{G}. There exists a mapping $ f:E(G_0)\rightarrow [{r}]\times [{r}] $ such that the following holds:
	\begin{itemize}
		\item  For any $g: V(G_0) \to [r]$, for any $ 1\le i<j \le k $, any subsets $ A \subseteq V_i , B \subseteq V_j $ with $ |A|,|B| \ge \lfloor n/t\rfloor $,   there is at least one edges $ e=xy $ with $ x\in A, y \in B $ such that $ f(e)=(g(x),g(y))$.
	\end{itemize} 
\end{lemma}
 \begin{proof}
 	Let $f: E(G_0) \to [r] \times [r]$ be a random mapping, where for each edge $e=xy$, and  $g: V(G_0) \to  [r] $, the probability that $f(e) = (g(x),g(y))$ is $1/r^2$.

 	 For two subsets $ A\subseteq V_i,B\subseteq V_j$ with  $i<j $, for $g: V(G_0) \to [r]$, we say the pair $ (A,B) $ is \emph{bad with respect to $g$}   if $ |A|= |B|=\lfloor n/t\rfloor $ and     there is no edges $ e=xy $ with $ x\in A, y \in B $ such that $ f(e)=(g(x),g(y)) $. We say $A,B$ is {\em bad} if $(A,B)$ is bad with respect to some $g: V(G_0) \to [r]$.
 	 To prove Lemma \ref{lem-split}, it suffices  to show that with positive probability, there is no bad pair. 
 	 
 	 By Lemma \ref{G}, for given $g: V(G_0) \to [r]$,  for each subsets $ A \subseteq V_i , B \subseteq V_j $ $ (i < j) $ with $ |A|,|B|= \lfloor n/t\rfloor $, there are at least $ \frac{1}{2}n^{1+\epsilon} $ edges between $ A $ and $ B $. For each  $a,b \in [r] $, and for each edge $e=xy$ with $x\in A$ and $y\in B$, the probability that $f(e) \ne (g(x),g(y))$ is $1- \frac{1}{r^2}$.  
 	 Thus the probability that $ (A,B) $ is bad with respect to $g$ is 
 	 $$\left(1-\frac{1}{{r}^2}\right)^{\frac{1}{2}n^{1+\epsilon}}, $$
 	 Let $ P $ be the probability that there exists a bad pair.  Then 
 	 $$ P\le q\binom{n}{\lfloor\frac{n}{k}\rfloor}^2{r}^{kn}\left(1-\frac{1}{{r}^2}\right)^{\frac{1}{2}n^{1+\epsilon}}<q(ek{r}^k)^{2n}\left(1-\frac{1}{{r}^2}\right)^{\frac{1}{2}n^{1+\epsilon}}. $$
 	 Assume that $ n $ is large enough so that
 	 $$ q(ek{r}^k)^{2n}\left(1-\frac{1}{{r}^2}\right)^{\frac{1}{4}n^{1+\epsilon}}<1. $$
 	 Hence
 	 $$ P< \left(1-\frac{1}{{r}^2}\right)^{\frac{1}{4}n^{1+\epsilon}}<1. $$
 	 Hence with positive probability, there is no bad pair, and the required  mapping $f$ exists. This completes the proof of Lemma \ref{lem-split}.  
 \end{proof}

Let $f: E(G_0) \to [r] \times [r]$ be the mapping in Lemma \ref{lem-split}.
Let $G'$ be the graph with vertex set $V(G_0) \times [r]$
in which $(x,s)$ is adjacent to $(y,t)$ if $e=xy \in E(G_0)$ and $x\in V_i,y\in V_j$ with $i<j$ and $f(e) = (s,t)$. 

Let $G$ be obtained from $G'$ by taking,  for  each $i=1,2,\ldots, q$ and for each vertex $v \in V_i$, one copy of $J_i$  and   identify the vertex set of this copy of $J_i$ with $\{v\} \times [r]$.  
 
Now we show that for appropriate chosen constants $n,t$, 
$G$ is not $\lambda'$-choosable. 

 Let $ C_1',C_2',...,C_p' $ be disjoint colour sets such that   $ |C_j'|=2k_j'-1 $ for $j=1,2,\ldots,p$. Let
$$ \mathcal{L}=\left\{\bigcup_{j=1}^p S_j: S_j\in \binom{C_j'}{k_j'}\right\} $$ 
Here $ \binom{C_j'}{k_j'} $ the family of all $ k_j' $-subsets of $ C_j' $. So each element of $\mathcal{L}$ is a $k'$-set of colours, where $k' = k'_1+k'_2+ \ldots +k'_p$.

 Let $$ t=2{r}|\mathcal{L}|k'.$$
 We construct  a $ \lambda' $-assignment of $ G $ as follows:
\begin{itemize}
	\item For each vertex $v$ of $G_0$, all the vertices in $\{v\} \times [r]$  is assigned the same list from $ \mathcal{L} $.   
	\item For each $i=1,2,\ldots,q$,  each list from $ \mathcal{L} $ is assigned to exactly $   \frac{n}{|\mathcal{L}|}   $ copies of $ J_i $ in $G_i$. (We assume that $n$ is chosen to be a multiple of $|\mathcal{L}|$).
\end{itemize}
Recall that $G_i$ is the subgraph of $G$ induced by $V_i \times [r]$, which consists of $n$ copies of $J_i$. 

It follows from the definition that $L$ is a $\lambda'$-assignment. We shall show that $G$ is not $L$-colourable, and hence $G$ is not $\lambda'$-choosable.

Assume to the contrary that   there is an $ L $-colouring $\phi$ of $ G $. For each index $ j\in \{1,2,...,p\} $, we say $ C_j' $ is \emph{occupied} by $ G_i $ if there are at least $ k_j' $ colours in $ C_j' $ such that each of them is used by at least $ \lceil n{r}/t\rceil $ vertices in $ G_i $. For each $ i \in \{1,2,...,q\}$, let
$$ N_i=\{j:C_j'\text{ is occupied by } G_i \} .$$

\begin{claim}
	For any $ i,i'\in \{1,2,...,q\} $ and $ i< i' $, we have $ N_i\cap N_{i'}=\emptyset $.
\end{claim}

\begin{proof}
	Assume $ N_i\cap N_{i'}\neq \emptyset $, say $ j\in N_i\cap N_{i'} $. By definition, there are at least $ k_j' $ colours in $ C_j' $ such that each of them is used by at least $ \lceil n{r}/t\rceil $ vertices in $ G_i $, and at least $ k_j' $ colours in $ C_j' $ such that each of them is used by at least $ \lceil n{r}/t\rceil $ vertices in $ G_{i'} $. As  $|C_j'|=2k_j'-1 $, there is a colour $ c\in C_j' $ used by at least $ \lceil n{r}/t\rceil $ vertices in $ G_{i} $ and also at least $ \lceil n{r}/t\rceil $ vertices in $ G_{i'} $. Thus there are at least $ \lfloor n/t\rfloor $ copies of $ J_i $ containing a vertex coloured by $ c $ in $ G_i $, and  at least $\lfloor n/t\rfloor $ copies of $ J_{i'} $ containing a vertex coloured by $ c $ in $ G_{i'} $.
	
	Let 
	\begin{eqnarray*}
	 A &=& \{v \in V_i: \text{ some vertex in $\{v\} \times [r]$ is coloured by $c$}\}, \\
	 B &=& \{v \in V_{i'}: \text{ some vertex in $\{v\} \times [r]$ is coloured by $c$}\}.
	\end{eqnarray*}
	Then $ |A|,|B|\ge \lfloor n/t\rfloor $. Let  $ g:  V(G_0) \rightarrow [r] $ be any mapping such that for all $x \in A \cup B$,    $ g(x)= a$ for some $a \in [r]$ such that $\phi (x,a) = c$. By Lemma \ref{lem-split},  there exists an edge $e=xy$ of $G_0$ such that $f(e) = (g(x), g(y))$. Hence $G$ has an edge connecting $(x,g(x))$ and $(y, g(y))$. But both $(x,g(x))$ and $(y, g(y))$ are coloured by $c$,  a contradiction.
\end{proof}

\begin{claim}
	For each index $ i\in \{1,2,...,q\} $, $ \sum_{j\in N_i}k_j'\ge k_i $.
\end{claim}

\begin{proof}
	For each $ j\notin N_i $, there is a set $ D_j $ of $ k_j' $ colours in $ C_j' $, and each colour in $ D_j $ is used by less than $ \lceil n{r}/t\rceil $ vertices in $ G_i $. Let $ L_0=\cup_{j=1}^{p}S_j^0\in \mathcal{L} $ be a list such that $ S_j^0=D_j $ for each $ j\notin N_i $ and $S_j^0$ is an arbitrary $k'_j$-subset of $C'_j$ for $j \in N_i$. 
	
	By the definition of $L$,  there exists $X \subseteq V_i$ such that 
	$$|X| \ge  \lfloor n/|\mathcal{L}|\rfloor $$
	and $$L(x,s)=L_0, \ \forall (x,s) \in X \times [r].$$ 
	 Let 
	 $$Z = \{(x,s) \in X \times [r]: \phi(x,s) \in   \cup_{j\notin N_i} D_j \}.$$
	 As each colour in  $\cup_{j\notin N_i} D_j$ is used by   less than $ \lceil n{r}/t\rceil $ vertices in $ G_i $, we conclude that 
	$$|Z| <  \left\lceil \frac{n{r}}{t}\right\rceil \sum_{j\notin N_i}k_j'<\left\lfloor \frac{n}{|\mathcal{L}|}\right\rfloor .$$
	So there exists $x \in X$ such that all vertices in 
	$\{x\} \times [r]$ are coloured by colours in   $ \cup_{j\in N_i} S_j^0$. Since $ J_i $ is not $ (k_i-1) $-colourable, we conclude that
	$$ |\bigcup_{j\in N_i} S_j^0|=\sum_{j\in N_i} k_j'\ge k_i. $$	
	
	Let $ \lambda''=\{k_1'',k_2'',...,k_q''\} $, where $ k_i''= \sum_{j\in N_i} k_j'$. Then $ \lambda'' $ is obtained from $\lambda$ by increasing some parts of $\lambda$, and $\lambda'$ is a refinement of $ \lambda'' $. Hence $ \lambda\le \lambda' $, which is in contrary to our assumption.	
\end{proof} 

Note that the structure of $G$ constructed in the proof of Theorem \ref{key} relies more on $\lambda$. The role of $\lambda'$ is only used in choosing $t$ and $n$. Thus the same proof actually proves the following stronger result.

\begin{theorem}
	\label{key2}
	Assume $\lambda$ and $\lambda_i$ ($i=1,2,\ldots, p$) are partitions of integers and $\lambda \not\le \lambda_i$ for each $i=1,2,\ldots, p$. Then for any positive integer $g$,  there exists a graph $G$ of girth $g$ which is $\lambda$-choosable, but not $\lambda_i$-choosable for $i=1,2,\ldots, p$. 
\end{theorem}

On the other hand, the following question remains open. 

\begin{question}
	\label{q1}
	Assume $\lambda_i \not\le \lambda$ for $i=1,2, \ldots, p$. Is it true that there exists a graph $G$ which is $\lambda_i$-choosable for $i=1,2,\ldots, p$ but not $\lambda$-choosable? 
	\end{question}

If the answer to Question \ref{q1} is "yes", then a natural next question is whether we can further require the graph $G$ to have large girth.


\begin{thebibliography}{10}
	
\bibitem{Alon} N.~Alon, A.~Kostochka, B.~Reiniger, D.W.~Douglas, and X.~Zhu, \emph{Coloring, sparseness, and girth}. Israel J. Math. 214(2016), 315-331.	

\bibitem{E} P.~Erd\H{o}s, \emph{Graph theory and probability}, Canadian Journal of Mathematics 11 (1959), 34-38.

\bibitem{Erdos} P.~Erd\H{o}s, A.L.~Rubin, H.~Taylor, \emph{Choosability in graphs}, in: Proc. West Coast Conf. on Combinatorics, Graph Theory and Computing, Congressus Numerantium XXVI, 1979, pp. 125-157.

\bibitem{KN} A.~V.~Kostochka, J.~Ne\v {s}et\v{r}il, \emph{Properties of Descartes’ Construction of Triangle-Free Graphs with High Chromatic Number}, Combinatorics, Probability and Computing 8 (1999), 467-472.


\bibitem{JZ} J.~Ne\v {s}et\v{r}il, X.~Zhu, \emph{On sparse graphs with given colorings and homomorphisms}, Journal of Combinatorial Theory, Series B 90 (2004) 161-172.


\bibitem{Viz} V.G.~Vizing, \emph{Coloring the vertices of a graph in prescribed colors}, Diskret. Analiz 29 Metody Diskret. Anal. v Teorii Kodov i Shem (1976) 3-10, p. 101, (in Russian).

\bibitem{Z} X.~Zhu, \emph{Uniquely H-colorable graphs with large girth}, J. Graph Theory 23 (1996) 33-41.

\bibitem{Zhu}  X.~Zhu, \emph{A refinement of choosability of graphs}. J. Combin. Theory, Ser. B 141 (2020) 143 - 164.


	
\end{thebibliography}
	\end{document}